\documentclass[11pt,twoside]{amsart}
\usepackage{amsfonts}
\usepackage{amsmath}
\usepackage{amssymb}
\usepackage[latin1]{inputenc}
\usepackage{graphicx}
\usepackage{mathrsfs}
\usepackage{cancel}

\newtheorem{theo}{Theorem}
\newtheorem{cor}{Corollary}
\newtheorem{lem}{Lemma}
\newtheorem{prop}{Proposition}

\theoremstyle{definition}
\newtheorem{defn}{Definition}

\theoremstyle{remark}
\newtheorem{rem}{\bf Remark\/}

\newtheorem{exple}{\bf Example\/}
\numberwithin{equation}{section}

\def\C{{\mathbb{C}}}
\def\N{{\mathbb{N}}}
\def\1{{\mathchoice {\rm 1\mskip-4mu l} {\rm 1\mskip-4mu l}{\rm 1\mskip-4.5mu l} {\rm 1\mskip-5mu l}}}
\newcommand{\ds}{\displaystyle}
\newcommand{\w}{\wedge}

\title[On the tangent cones to psh currents]{On the tangent cones to plurisubharmonic currents}
\author[N. Ghiloufi]{Noureddine Ghiloufi}
\email{noureddine.ghiloufi@fsg.rnu.tn}
\author[K. Dabbek]{Khalifa Dabbek}
\email{khalifa.dabbek@fsg.rnu.tn}

\address{Department of Mathematics\\ Faculty of sciences of Gab\`es \\ University of Gab\`es \\ 6072 Gab\`es Tunisia.}
\begin{document}
\maketitle

\begin{abstract}
    In this paper, we study the existence of the tangent cone to a positive plurisubharmonic or plurisuperharmonic current with a suitable condition. Some Estimates of the growth of the Lelong functions associated to the current  and to its $dd^c$ are given to ensure the existence of the blow-up of this current. A second proof for the existence of the tangent cone is derived from these estimates.\\

    \textbf{Sur les cônes tangents au  courants plurisousharmoniques.}\\
    \textsc{R\'esum\'e.} Dans cet article, nous \'etudions l'existence du c\^one tangent \`a un courant positif plurisousharmonique ou plurisurharmonique avec une condition convenable. Des estimations de croissance des fonctions de Lelong associ\'ees au courant et \`a son $dd^c$ sont donn\'ees pour assurer l'existence du rel\`evement de ce courant. Une  deuxi\`eme preuve de l'existence du c\^one tangent se d\'eduit de ces estimations.
\end{abstract}

\section{Introduction}
    Let $T$ be a positive current of bidimension $(p,p)$ on a neighborhood $\Omega$ of 0 in $\C^n$, $0<p<n$, and $h_a$ be the complex dilatation on $\C^n$ ($h_a(z)=az$)  with $a\in\C^*$. In this paper we study the existence of the weak limit of the family of currents $(h_a^\star T)_a$ when $|a|$ tends to 0. A such limit is called tangent cone to $T$. The case of analytic sets was studied by Thie in 1967 and then by King in 1971. Thus they prove that the tangent cone to the current $[A]$ (current of integration over the analytic set $A$) is given by the current of integration over the geometric tangent cone to the analytic set $A$. However, this statement is not true in  case of positive closed currents and a counterexample was given by Kiselman where he constructed a psh function $u$ such that the current $dd^cu$ doesn't have a tangent cone. For this reason, to ensure the existence of the tangent cone we need some conditions. For closed positive currents, Blel, Demailly and Mouzali gave two independent conditions, where each one ensures the existence of the tangent cone. In this paper, we show that the second condition, condition  $(b)$ in \cite{Bl-De-Mo}, is even sufficient in the case of positive plurisubharmonic ($dd^cT\geq0$) or plurisuperharmonic ($dd^cT\leq0$) currents. Precisely we have\\

    \begin{theo} \label{th1} \textbf{(Main result)} Let $T$ be a positive plurisubharmonic (resp. pluri\-superharmonic) current of bidimension $(p,p)$ on $\Omega$, $0<p<n$. Then the tangent cone to $T$ at 0 exists if, for $r_0>0$, we have
    $$\int_0^{r_0}\frac{\nu_T(r)-\nu_T(0)}r dr <+\infty$$
    resp.
    $$\int_0^{r_0}\frac{\nu_{dd^cT}(t)}t dt >-\infty\quad \hbox{and}\quad \int_0^{r_0}\frac{|\nu_T(r)-\nu_T(0)|}r dr <+\infty.$$
    \end{theo}
    We start our paper by giving some preliminary results. Next, we give a direct proof of the main result. Finely, we study the problem of restriction of positive currents along analytic sets and we conclude a second (partial) proof of the main result.
    \subsection{Lelong numbers}
        Let now recall some notations and preliminary results useful in the following.\\
        For every $r>0,\ r_2>r_1>0$ and $z_0\in\C^n$, we set
        $$\begin{array}{l}
            \ds B(z_0,r):=\left\{z\in\C^n;\ |z-z_0|<r\right\}\\
            B(z_0,r_1,r_2):=\{z\in\C^n;\ r_1\leq|z-z_0|<r_2\}=B(z_0,r_2)\smallsetminus B(z_0,r_1) \\
            \ds \beta_{z_0}:=dd^c|z-z_0|^2=\frac{i}{2\pi}\partial\overline{\partial}|z-z_0|^2,\quad \alpha_{z_0}:=dd^c\log|z-z_0|^2.
        \end{array}$$
        When $z_0=0$, we omit $z_0$ in previous notations and we use only $B(r)$, $B(r_1, r_2)$, $\beta$ and $\alpha$ instead of $B(0,r)$, $B(0,r_1,r_2)$, $\beta_0$ and $\alpha_0$ respectively.\\

        Let $T$ be a positive plurisubharmonic or plurisuperharmonic current of bidimension $(p,p)$ on $\Omega$ and $z_0\in\Omega$. Let $R>0$ such that $B(z_0,R)\subset\subset\Omega$. For all $0<r<R$, we set  $\nu_T(z_0,r)=\frac1{r^{2p}}\int_{B(z_0,r)}T\w\beta_{z_0}^p$ the projective mass of $T$. The well-known following lemma will be used frequently in the hole of this paper.
        \begin{lem}\label{lem1}(Lelong-Jensen formula)
            Let $S$ be a positive plurisubharmonic or plurisuperharmonic current of bidimension $(p,p)$ on $\Omega$ and $z_0\in\Omega$. Then, for all $0<r_1<r_2<R$,
            \begin{equation}\label{eq1.1}
                \begin{array}{lcl}
                    \nu_S(z_0,r_2)-\nu_S(z_0,r_1) & = &\ds \frac1{r_2^{2p}}\int_{B(z_0,r_2)} S\w\beta_{z_0}^p -\frac1{r_1^{2p}} \int_{B(z_0,r_1)}S\w\beta_{z_0}^p\\
                    & = & \ds \int_{r_1}^{r_2}\left(\frac1{t^{2p}} -\frac1{r_2^{2p}} \right)tdt\int_{B(z_0,t)}dd^cS\w\beta_{z_0}^{p-1}\\
                    & &\hfill \ds +\left(\frac1{r_1^{2p}}-\frac1{r_2^{2p}}\right)\int_0^{r_1}tdt\int_{B(z_0,t)} dd^cS\w \beta_{z_0}^{p-1}\\
                    & &\hfill \ds+\int_{B(z_0,r_1,r_2)}S\w\alpha_{z_0}^p.
                \end{array}
            \end{equation}
        \end{lem}
        According to Lemma \ref{lem1}, if $T$ is  positive plurisubharmonic then $\nu_T(z_0,.)$ is a non-negative increasing function on $]0,R[$, so  the Lelong number $\nu_T(z_0):=\lim_{r\to0^+} \nu_T(z_0,r)$ of $T$ at $z_0$ exists.\\
        For positive plurisuperharmonic currents, the existence of Lelong numbers was treated  by the first author and he proved that Lelong numbers do not depend on the system of coordinates. We cite the main result of \cite{Gh}.
        \begin{theo}\label{th2}
            Let $T$ be a positive plurisuperharmonic current of  bidimension $(p,p)$ on $\Omega$, $0<p<n$, and $z_0\in\Omega$. We assume that $T$  satisfies condition $(C)_{z_0}$ given by:
            $$(C)_{z_0}:\qquad\int_0^{r_0}\frac{\nu_{dd^cT}(z_0,t)}tdt>-\infty$$ for some $0<r_0\leq R$. Then, the Lelong number $\nu_T(z_0)$ of $T$ at $z_0$ exists.
        \end{theo}
        \begin{proof}
            For every $0<r<R$, we set
            $$\Lambda_{z_0}(r)=\nu_T(z_0,r)+\int_0^r\left(\frac{t^{2p}}{r^{2p}}-1\right)\frac{\nu_{dd^cT}(z_0,t)}{t}dt.$$
            Thanks to condition $(C)_{z_0}$ and using the fact that $\nu_{dd^cT}(z_0,.)$ is non-positive on $]0,R[$, one can deduce that $\Lambda_{z_0}$ is well defined and non-negative on $]0,R[$.\\
            For $0<r_1<r_2<R$, Lemma \ref{lem1} gives
            $$\begin{array}{lcl}
                    \Lambda_{z_0}(r_2)-\Lambda_{z_0}(r_1) & = &\ds\nu_T(z_0,r_2)- \nu_T(z_0,r_1)+ \frac1{r_2^{2p}}\int_0^{r_2}t^{2p-1} \nu_{dd^cT}(z_0,t)dt \\
                    & &\hfill \ds -\frac1{r_1^{2p}}\int_0^{r_1}t^{2p-1} \nu_{dd^cT}(z_0,t)dt -\int_{r_1}^{r_2}\frac{\nu_{dd^cT}(z_0,t)}{t}dt\\
                    &=&\ds \int_{B(z_0,r_1,r_2)}T\w\alpha_{z_0}^p\geq 0.
                \end{array}$$
            Therefore, $\Lambda_{z_0}$ is a non-negative increasing function on $]0,R[$, and this implies the existence of the  limit $\varrho:=\lim_{r\to0^+} \Lambda_{z_0}(r)$. The hypothesis of integrability of $\nu_{dd^cT}(z_0,t)/t$ and the fact that $(t^p/r^p-1)$ is uniformly  bounded give $$\ds\lim_{r\to0^+}\int_0^r\left(\frac{t^{2p}}{r^{2p}}-1\right)\frac{\nu_{dd^cT}(z_0,t)}{t}dt=0.$$
            Hence, $\varrho=\lim_{r\to0^+}\Lambda_{z_0}(r)=\lim_{r\to0^+}\nu_T(z_0,r)=\nu_T(z_0)$.
        \end{proof}
        The following example proves that Condition $(C)$  is not necessary in Theorem \ref{th2} for the existence of Lelong number.
        \begin{exple}\label{exple1}
            Let $T_0=du\w d^cu$ where $u(z)=\log|z|^2$. Then $T_0$ is a positive current of bidimension $(1,1)$ on $\C^2$. Furthermore one has $dd^cT_0=-(dd^cu)^2=-\delta_0$ (Dirac) is negative on $\C^2$ and $\nu_{dd^cT_0}(0)=-1$, so Condition $(C)_0$ is not satisfied. In the other hand, a simple computation shows that
            $$\nu_{T_0}(r)=\frac1{4\pi^2 r^2}\int_{|z|<r}\frac1{|z|^2}idz_1\w d\overline{z}_1\w idz_2\w d\overline{z}_2=\frac14.$$
        \end{exple}
        An open problem arises from this part which is to study  the set $\mathscr E_\infty(T)$ of points $z$ in $\Omega$ for which the Lelong number of $T$ at $z$ doesn't exist. If $T$ is positive plurisubharmonic  then $\mathscr E_\infty(T)$ is empty, but if $T$ is positive plurisuperharmonic then $\mathscr E_\infty(T)$ can be non-empty, for example $\mathscr E_\infty(T_1)=\{0\}$ where $T_1=-\log|z_1|^2[z_2=0]$ on $\C^2$.\\
        We remark that if $T$ is positive plurisuperharmonic, then $\mathscr E_\infty(T)\subset \mathscr F_\infty(T)$ where
        $$\begin{array}{lcl}
            \mathscr F_\infty(T)& := &\ds \left\{z\in\Omega;\ \frac{\nu_{dd^cT}(z,t)}t\not\in L^1(\vartheta(0))\right\} \\
            & = & \ds \left\{z\in \Omega;\ \nu_{dd^cT}(z)<0\right\}\cup\left\{z\in \mathscr F_\infty(T);\ \nu_{dd^cT}(z)=0\right\}\\
            & =: & \mathscr F_\infty^1(T)\cup \mathscr F_\infty^2(T).
          \end{array}$$
        The subset $\mathscr F_\infty^1(T)$ is  pluripolar in $\Omega$. Indeed, thanks to Siu's theorem,
        $$\mathscr F_\infty^1(T)=\bigcup_{j\in\N^*}\left\{z\in \Omega;\ \nu_{dd^cT}(z)\leq-\frac1j\right\}$$ is a countable union of analytic sets, because $dd^cT$ is a negative closed current. For the second subset $\mathscr F_\infty^2(T)$, we conjecture that it is also pluripolar.

    \subsection{A Structure theorem}
        In this part, we study a geometric structure of the support of a positive plurisuperharmonic current in $\Omega$. All results given are available in any complex manifold of dimension $n$.\\
        Our aim here is to prove the following theorem:
        \begin{theo}\label{th3}
            Let $T$ be a positive plurisuperharmonic current of  bidimension $(p,p)$ on an open set $\Omega$ of $\C^n$ ($1\leq p\leq n-1$) such that the function $t\mapsto \frac{\nu_{dd^cT}(z,t)}t$ is locally uniformly integrable in neighborhood of points of $X=Supp(T)$. Assume that there exists a real number $\delta>0$ such that the level-set $E_\delta:=\{z\in\Omega;\ \nu_T(z)\geq \delta\}$ is dense in $X$. Then $X$ is a complex subvariety of pure dimension $p$ of $\Omega$ and there exists a weakly plurisubharmonic negative function $\varphi$ on $X$ such that $T=-\varphi[X]$.
        \end{theo}
        A similar result was given by Dinh-Lawrence \cite{Di-La} in the case of positive plurisubharmonic currents.\\
        To prove this theorem we have to recall some results.
        \begin{defn}
            Let $Z$ be a closed subset of $\Omega$ and $p$ an integer, $0<p<n$ ($n\geq2$). We say that $Z$ is $p-$pseudoconcave in $\Omega$ if for every open set $\mathcal U\subset\subset \Omega$ and every holomorphic map $f$ from a neighborhood of $\overline{\mathcal U}$ into $\C^p$ we have $f(Z\cap \mathcal U)\subset \C^p\smallsetminus V$ where $V$ is the bounded component of $\C^p\smallsetminus f(Z\cap \partial \mathcal U)$.
        \end{defn}
        Pseudoconcave sets  were studied by Forn{\ae}ss-Sibony \cite{Fo-Si}, Dinh-Lawrence \cite{Di-La} and others. It was shown that the support of a positive plurisuperharmonic current is an example of pseudoconcave sets, Precisely we have:
        \begin{lem}(See \cite{Di-La, Fo-Si})\label{lem2}
            Let $T$ be a positive plurisuperharmonic current of  bidimension $(p,p)$ on an open set $\Omega$ of $\C^n$ ($1\leq p\leq n-1$). Then $X:=Supp(T)$ is  $p-$pseudoconcave in $\Omega$.
        \end{lem}
        The fundamental tool in the proof of Theorem \ref{th3} is the following lemma.
        \begin{lem}(See \cite{Di-La})\label{lem3}
            Let $\Omega$ be a complex manifold of dimension $n\geq2$ and $X$ a $p-$pseudoconcave subset of $\Omega$. Let $K$ be a compact subset of $\Omega$ which admits a Stein neighborhood. Assume that the $2p-$dimensional Hausdorff measure of $X \smallsetminus K$ is locally finite in $\Omega \smallsetminus K$. Then $X$ is a complex subvariety of pure dimension $p$ of $\Omega$.
        \end{lem}
        Now we can prove Theorem \ref{th3}.
        \begin{proof}
            It is sufficient to prove that $X$ is a complex subvariety of pure dimension $p$ of $\Omega$. Thanks to lemmas \ref{lem2} and  \ref{lem3}, we have to prove that $X$ has locally finite $\mathscr H^{2p}$ Hausdorff measure. In fact we prove that $X=E_\delta$ and $E_\delta$ has locally finite $\mathscr H^{2p}$ measure. For the last affirmation, we remark that for every $z\in E_\delta$ one has
            $$\delta\leq \nu_T(z)=\lim_{t\to0}\frac{\sigma_T(B(z,t))}{t^{2p}}$$ which proves that $E_\delta$ has locally finite $\mathscr H^{2p}$ measure. Since $E_\delta\subset X$, To prove $X=E_\delta$, it suffice to prove  $E_\delta$ is closed in $\Omega$. For this, Let $(\xi_j)_j\subset E_\delta$ and $\xi_j\underset{j\to+\infty}\longrightarrow \xi\in\Omega$. Fix an $r_0>0$ such that $B(\xi,r_0)\subset \Omega$ and $j_0$ such that $\xi_j\in B(\xi,r_0)$ for all $j\geq j_0$. Since $\sigma_T:= T\w\beta^p$ is a positive measure, then for every $0<r<r_0$ and for every $j\geq j_1$ (for which $0<|\xi-\xi_j|<r$), we have $\sigma_T(B(\xi,r))\geq \sigma_T(B(\xi_j,r-|\xi-\xi_j|))$. One has
            $$\delta\leq \nu_T(\xi_j)\leq \Lambda_{\xi_j}(s_j):= \nu_T(\xi_j,s_j)+ \int_0^{s_j}\left(\frac{t^{2p}}{s_j^{2p}}-1\right)\frac{\nu_{dd^cT}(\xi_j,t)}{t}dt.$$
            where $s_j=r-|\xi-\xi_j|$. So,
            $$\begin{array}{lcl}
                  \ds \sigma_T(B(\xi,r))&\geq &\ds \delta s_j^{2p}-s_j^{2p} \int_0^{s_j}\left(\frac{t^{2p}}{s_j^{2p}}-1\right)\frac{\nu_{dd^cT}(\xi_j,t)}{t}dt \\
                   & \geq &\ds \delta s_j^{2p}-s_j^{2p} \mathscr I_j.
                \end{array}$$
            We set $\mathscr I_j=\mathscr I_{j,1}+\mathscr I_{j,2}$ where
            $$\begin{array}{lcl}
                0 \leq \mathscr I_{j,1} & = & \ds\int_0^{|\xi-\xi_j|}\left(\frac{t^{2p}}{(r-|\xi-\xi_j|)^{2p}}-1\right)\frac{\nu_{dd^cT}(\xi_j,t)}{t}dt \\
                 & \leq &\ds -\int_0^{|\xi-\xi_j|}\frac{\nu_{dd^cT}(\xi_j,t)}{t}dt \underset{\xi_j\to\xi}\longrightarrow0.
              \end{array}$$
            and
            $$\begin{array}{lcl}
                \mathscr I_{j,2} & = & \ds\int_{|\xi-\xi_j|}^{r-|\xi-\xi_j|}\left(\frac{t^{2p}}{(r-|\xi-\xi_j|)^{2p}}-1\right) \frac{\nu_{dd^cT}(\xi_j,t)}{t}dt \\
                 & \leq &\ds \int_{|\xi-\xi_j|}^{r-|\xi-\xi_j|}\left(\frac{t^{2p}}{(r-|\xi-\xi_j|)^{2p}}-1\right) \frac{\sigma_{dd^cT}(B(\xi,t+|\xi-\xi_j|))}{t^{2p-1}}dt\\
                 & \underset{\xi_j\to\xi}\longrightarrow & \ds \int_0^r\left(\frac{t^{2p}}{r^{2p}}-1\right) \frac{\sigma_{dd^cT}(B(\xi,t))}{t^{2p-1}}dt =\ds \int_0^r\left(\frac{t^{2p}}{r^{2p}}-1\right) \frac{\nu_{dd^cT}(\xi,t)}{t}dt.
              \end{array}$$
            Hence
            $$\nu_T(\xi,r)\geq \delta-\int_0^r\left(\frac{t^{2p}}{r^{2p}}-1\right) \frac{\nu_{dd^cT}(\xi,t)}{t}dt.$$
            If $r\to0$, we obtain  $\nu_T(\xi)\geq \delta$; so $\xi\in E_\delta$ which proves that $E_\delta$ is closed.\\
            We can conclude from previous computations  the upper-semi-continuity of $\nu_T$ on $\Omega$.
        \end{proof}

\section{Proof of the main result}
    In this part, we study the existence of the tangent cone to positive plurisubharmonic or plurisuperharmonic currents on an open neighborhood  $\Omega$ of 0 in $\C^n$. The principal result will be proved partially with a different way in the third section.\\

    In the following, we will use $\mathscr P^+_p(\Omega)$ (resp. $\mathscr P^-_p(\Omega)$) to indicate the set of positive plurisubharmonic currents (resp. positive plurisuperharmonic currents satisfying condition $(C)_0$) of bidimension $(p,p)$ on $\Omega$ where $0<p<n$.\\

    \noindent\textbf{Theorem \ref{th1}. (Main result)}
        \textit{Let $T\in\mathscr P_p^\pm(\Omega)$.  Then the tangent cone to $T$ exists when
        $$\int_0^{r_0}\frac{|\nu_T(r)-\nu_T(0)|}rdr<+\infty$$
        where $r_0$ is a positive real such that $B(r_0)\subset\subset\Omega$.}\\

    This theorem is due to Blel-Demailly-Mouzali in case of positive closed currents. In \cite{Ha}, Haggui proved the same result for $T\in\mathscr P_p^+(\Omega)$. His proof is based on the potential  current associated to $dd^cT$. Here, we present a proof which is different from Haggui's one.
    \begin{rem}
        Condition $(C)_0$ is not necessary in Theorem \ref{th1}. In fact the current $T_0$ of Example \ref{exple1} admits a tangent cone and doesn't satisfy Condition $(C)_0$. Indeed $h_a^\star T_0=T_0$ ($T_0$ is conic on $\C^2$).
    \end{rem}
    \begin{proof}
        Let $T\in\mathscr P_p^+(\Omega)$ (resp. $T\in\mathscr P_p^-(\Omega)$). Using $h_a^\star T$ in equality (\ref{eq1.1}) and the equality $\nu_{h_a^\star T}(r)=\nu_T(|a|r)$ for all $|a|<r_0/r$,  we find
        \begin{equation}\label{eq2.1}
            \int_{B(r)}h_a^\star T\w\beta^p\leq \nu_T(r_0)r^{2p},\quad \forall\; |a|\leq \frac{r_0}r.
        \end{equation}
        resp.
        \begin{equation}\label{eq2.2}
            \int_{B(r)}h_a^\star T\w\beta^p\leq \Lambda_0(r_0)r^{2p},\quad \forall\; |a|\leq \frac{r_0}r.
        \end{equation}
        In both cases, equations  (\ref{eq2.1}) and (\ref{eq2.2}) give the mass of $(h_a^\star T)$ is uniformly small in the neighborhood of 0. Hence $(h_a^\star T)$ converges weakly on $\C^n$ if and only if it converges weakly in the neighborhood of every point $z^0\in\C^n\smallsetminus \{0\}$. After a suitable dilatation and a unitary changement of coordinates, we can assume that $z^0=(0,...,0,z_n^0)$ where $1/2<z_n^0<1$. We use projective coordinates and we set
        $$w_1=\frac{z_1}{z_n},...,\ w_{n-1}=\frac{z_{n-1}}{z_n},\ w_n=z_n$$
        and $$T=2^{-q}i^{q^2}\sum_{|I|=|J|=q}T_{I,J}dw_I\w d\overline{w}_J$$ where $q=n-p$. The dilatation $h_a$ is written as $h_a:\ w=(w',w_n)\mapsto (w',aw_n)$ with $w'=(w_1,...,w_{n-1})$. We verify that the coefficients $T_{I,J}^a$ of $h_a^\star T$ are given by
        \begin{equation}\label{eq2.3}
            T_{I,J}^a(w)=\left\{
                \begin{array}{lcl}
                    T_{I,J}(w',aw_n)& if & n\not\in I,\ n\not\in J \\
                    a\;T_{I,J}(w',aw_n)& if & n\in I,\ n\not\in J \\
                    \overline{a}\;T_{I,J}(w',aw_n)& if & n\not\in I,\ n\in J \\
                    |a|^2\;T_{I,J}(w',aw_n)& if & n\in I,\ n\in J
                \end{array}\right.
        \end{equation}
        The proof of the main result when $T\in\mathscr P_p^-(\Omega)$ is similar to the case  $T\in\mathscr P_p^+(\Omega)$ with some simple modification, for this reason we will continuous this proof with $T\in\mathscr P_p^+(\Omega)$.\\
        We need the following lemma:
        \begin{lem}\label{lem4}
            Let $U$ be the neighborhood of $z^0$ given by
            $$U=\{z\in\C^n;\ |z|<1,\ 1/2<|z_n|<1\}\subset B(1/2,1).$$ We consider the two functions $\gamma_T(r)=\nu_T(r)-\nu_T(r/2)$ and $\gamma_{dd^cT}(r)=\nu_{dd^cT}(r)-\nu_{dd^cT}(r/2)$ defined on $]0,R[$. For $r_0<R$, there exist three positive constants $C_1,\ C_2$ and $C_3>0$ such that for $|a|<r_0$, the measure $T_{I,J}^a$ satisfies the following estimates
            \begin{equation}\label{eq2.4}
                \int_U|T_{I,J}^a|\leq\left\{
                    \begin{array}{ll}
                        C_1& \hbox{for all }I,\ J\\
                        C_2\left(\gamma_T(|a|)+\gamma_{dd^cT}(|a|)\right)& if\ n\in I,\ and\ n\in J\\
                        C_3\sqrt{\gamma_T(|a|)+\gamma_{dd^cT}(|a|)}& if\ n\in I,\ or\ n\in J
                    \end{array}\right.
            \end{equation}
        \end{lem}
        The proof of this lemma will be done later, so we can now continuous our proof.\\
        Thanks to Lemma \ref{lem4}, $T_{I,J}^a$ tends to 0 in mass for $I$ or $J$ containing $n$, hence, to finish the proof it suffice to study the weak convergence of measures $T_{I,J}^a$ when $n\not\in I$ and $n\not\in J$.\\
        Let $\varphi\in\mathscr D(U)$. For $n\not \in I=\{i_1,...,i_q\},\ n\not\in J=\{j_1,...,j_q\}$, we set
        $$f_{I,J}(a)=\int_UT_{I,J}^a(w)\varphi(w)d\tau(w)=\int_UT_{I,J}(w',aw_n)\varphi(w)d\tau(w).$$
        $f_{I,J}$ is $\mathcal C^\infty$ on $D^*(0,R):=\{a\in\C;\ 0<|a|<R\}$ and it is bounded in a neighborhood of 0. The problem is to show that  $f_{I,J}(a)$ admits a limit when $a\to0$. The idea is to estimate $\Delta f_{I,J}$ in a neighborhood of 0. We have
        $$\frac{\partial ^2 f_{I,J}}{\partial a\partial \overline{a}}(a) =\int_U|w_n|^2\frac{\partial ^2 T_{I,J}}{\partial w_n\partial \overline{w}_n}(w',aw_n) \varphi(w)d\tau(w).$$
        We remark that the coefficient of $dw_{I\cup\{ n\}}\w d\overline{w}_{J\cup\{ n\}}$ in the expression of $dd^cT$ is
        $$\begin{array}{lcl}
            \ds(dd^cT)_{I\cup \{n\},J\cup \{n\}} & = & \ds (-1)^q\frac{\partial ^2 T_{I,J}}{\partial w_n\partial \overline{w}_n}+\sum_{k,s=1}^q (-1)^{k+q+s-2}\frac{\partial ^2 T_{I(k),J(s)}}{\partial w_{i_k}\partial \overline{w}_{j_s}}\\
            & &\ds \hfill+ \sum_{s=1}^q (-1)^{s-1}\frac{\partial ^2 T_{I,J(s)}}{\partial w_n\partial \overline{w}_{j_s}}+\sum_{k=1}^q (-1)^{k-1}\frac{\partial ^2 T_{I(k),J}}{\partial w_{i_k}\partial \overline{w}_n}
        \end{array}$$
        where $I(k)=I\smallsetminus\{i_k\}\cup\{n\}$ and $J(s)=J\smallsetminus\{j_s\}\cup\{n\}$. It follows from equality(\ref{eq2.3}), that
        $$\begin{array}{lcl}
            \ds \frac{\partial ^2 f_{I,J}}{\partial a\partial \overline{a}}(a)& = &\ds (-1)^q \int_U\frac{|w_n|^2}{|a|^2}(dd^cT)_{I\cup \{n\},J\cup \{n\}}^a\varphi(w)d\tau(w)\\
            & &\ds +\sum_{k,s=1}^q (-1)^{k+s-1}\int_U\frac1{|a|^2} T_{I(k),J(s)}^a\frac{\partial ^2 \varphi}{\partial w_{i_k}\partial \overline{w}_{j_s}}d\tau(w)\\
            & &\ds +\sum_{k=1}^q (-1)^{q+k}\int_U\frac1a T_{I(k),J}^a\frac{\partial ^2 \varphi}{\partial w_{i_k}\partial \overline{w}_n}d\tau(w)\\
            &  & \ds + \sum_{s=1}^q (-1)^{q+s}\int_U\frac1{\overline{a}} T_{I,J(s)}^a\frac{\partial ^2 \varphi}{\partial w_n\partial \overline{w}_{j_s}}d\tau(w).
        \end{array}$$
        Thanks to lemma \ref{lem4}, one has
        $$\begin{array}{lcl}
            \ds\left|\frac{\partial ^2 f_{I,J}}{\partial a\partial \overline{a}}(a)\right| & \leq  & \ds C_1\frac{\gamma_{dd^cT}(|a|)}{|a|^2}+C_2\frac{\gamma_T(|a|)+\gamma_{dd^cT}(|a|)}{|a|^2}\\
            & &\ds\hfill+C_3 \frac{\sqrt{\gamma_T(|a|)+\gamma_{dd^cT}(|a|)}}{|a|}\\
            & \leq  & C \left(\frac{\gamma_T(|a|)+\gamma_{dd^cT}(|a|)}{|a|^2} +\frac{\sqrt{\gamma_T(|a|)+\gamma_{dd^cT}(|a|)}}{|a|}\right)=C \psi(|a|).
        \end{array}$$
        Thanks to \cite[lemme 3.6]{Bl-De-Mo}, $f_{I,J}(a)$ admits a limit at 0 if $\psi$ satisfies
        $$\int_0^{r_0}r|\log r|\psi(r)dr<+\infty.$$
        A simple computation (see \cite{Bl-De-Mo}) shows that
        $$\int_0^{r_0}\frac{\gamma_T(r)+\gamma_{dd^cT}(r)}r|\log r|dr<+\infty $$
        is equivalent to $$\int_0^{r_0}\frac{\nu_T(r)-\nu_T(0)}rdr<+\infty \hbox{ and } \int_0^{r_0}\frac{\nu_{dd^cT}(r)}rdr<+\infty.$$
        those conditions are cited in the hypothesis of the main result. Hence, we have
        \begin{equation}\label{eq2.5}
            \int_0^{r_0}\frac{\gamma_T(r)+\gamma_{dd^cT}(r)}r|\log r|dr<+\infty.
        \end{equation}
        Thanks to Cauchy-Schwarz inequality, (\ref{eq2.5}) gives
        $$\begin{array}{lcl}
            \ds\int_0^{r_0}\sqrt{\gamma_T(r)+\gamma_{dd^cT}(r)}|\log r|dr & \leq & \ds \left( \int_0^{r_0}\frac{\gamma_T(r)+\gamma_{dd^cT}(r)}r|\log r|dr\right)^{1/2}\\
            & &\ds \hfill\times \left( \int_0^{r_0}r|\log r|dr\right)^{1/2}\\
            &<&+\infty.
        \end{array}$$
        Therefore,
        $$\int_0^{r_0}r|\log r|\psi(r)dr<+\infty$$
        which completes the proof of Theorem \ref{th1}.
    \end{proof}
    To prove Lemma \ref{lem4}, we need Demailly's inequality: \emph{If $$S=2^{-q}i^{q^2}\sum_{|I|=|J|=q}S_{I,J}dw_I\w d\overline{w}_J$$ is a positive $(q,q)-$current then for all $(\lambda_1,...,\lambda_n)\in]0,+\infty[^n$ we have
    \begin{equation}\label{eq2.6}
        \lambda_I\lambda_J|S_{I,J}|\leq 2^q\sum_{M\in\mathscr M_{I,J}}\lambda_MS_{M,M}
    \end{equation}
    where $\lambda_I=\lambda_{i_1}...\lambda_{i_q}$ if $I=(i_1,...,i_q)$ and the sum is taken over the set of $q-$index  $\mathscr M_{I,J}=\{M;\ |M|=q,\ I\cap J\subset M\subset I\cup J\}$.}\\

    Now, we can prove Lemma \ref{lem4}.
    \begin{proof}\
        \begin{itemize}
          \item The set $\overline{U}$ is compact and the $(1,1)-$forme $\beta$ is smooth and positive, so we have $\beta\geq C_4dd^c|w|^2$ on $U$. Inequality (\ref{eq2.1}), with $r=1$, implies $ \int_UT_{I,I}^a\leq C_5$ uniformly to $a$ for $|a|<r_0$. Demailly's inequality (\ref{eq2.6}), with the choice $\lambda_1=...=\lambda_n=1$, gives
        $$\int_U|T_{I,J}^a|\leq C_6\sum_{M\in\mathscr M_{I,J}}\int_UT_{M,M}^a\leq C_1$$
         so the first estimate in (\ref{eq2.4}) is proved.
          \item To prove the second estimate, we remark that we have $\alpha\geq C_7 \beta'$ on $U$ where $\beta'=dd^c|w'|^2$. Indeed, $\alpha=dd^c\log(1+|w'|^2)\geq \frac1{(1+|w'|^2)^2}\beta'\geq \frac14\beta'$ on $U$. Hence
              $$\begin{array}{lcl}
                  \ds\int_U\sum_{I\ni n}T_{I,I}^a & = & \ds \int_U h_a^\star T\w (dd^c|w'|^2)^p \\
                  & \leq &\ds C_8\int_U h_a^\star T\w \alpha^p \leq \ds C_8\int_{B(1/2,1)} h_a^\star T\w \alpha^p.
              \end{array}$$
              Thanks to Lelong-Jensen formula, with $r_2=1$ and $r_1=1/2$, one has
              $$\begin{array}{lcl}
                  \ds\int_U\sum_{I\ni n}T_{I,I}^a& \leq& \ds C_8\int_{B(1/2,1)} h_a^\star T\w \alpha^p\\
                  & \leq & \ds C_8\left[\nu_T(|a|)-\nu_T(|a|/2)-\int_{\frac12}^1\left(\frac1{t^{2p}} -1 \right)t^{2p-1}\nu_{dd^c(h_a^\star T)}(t)dt\right.\\
                  & & \hfill \ds-\left.\left(\frac1{2^{2p}}-1\right) \int_0^{\frac12} t^{2p-1}\nu_{dd^c(h_a^\star T)}(t)dt\right]\\
                  & \leq & \ds C_8\left[\nu_T(|a|)-\nu_T(|a|/2)-\int_{\frac12}^1\frac{\nu_{dd^cT}(|a|t)}t dt\right. \\
                  & & \hfill\ds \left.+\int_0^1t^{2p-1}\nu_{dd^cT}(|a|t)dt\right]\\
                  & \leq & \ds C_8(\nu_T(|a|)-\nu_T(|a|/2))+C_9(\nu_{dd^cT}(|a|)-\nu_{dd^cT}(|a|/2))\\
                  & \leq & \ds C_8\gamma_T(|a|)+C_9\gamma_{dd^cT}(|a|)
              \end{array}$$
              because $\nu_{dd^cT}$ is a non-negative increasing function. The second estimate is proved for $I=J\ni n$.\\
              for the general case, $I,J\ni n$, we use Demailly's inequality (\ref{eq2.6}) with $\lambda_1=...=\lambda_n=1$, to obtain
              $$\int_U|T_{I,J}^a|\leq C_{10}\sum_{M\in\mathscr M_{I,J}}\int_UT_{M,M}^a\leq C_2(\gamma_T(|a|)+\gamma_{dd^cT}(|a|))$$
              and the second estimate is proved.
          \item For the third estimate, it suffice to assume that $n\in I$ and $n\not\in J$. Again thanks to Demailly's inequality (\ref{eq2.6}), with $\lambda_1=...=\lambda_{n-1}=1$ and $\lambda_n>0$, we have
              $$\begin{array}{lcl}
                  \ds\lambda_n\int_U |T_{I,J}^a| &\leq &\ds  C_{11}\int_U\left(\sum_{n\not\in M\in\mathscr M_{I,J}}T_{M,M}^a +\lambda_n^2\sum_{n\in M\in\mathscr M_{I,J}}T_{M,M}^a \right) \\
                  & \leq & C_{12}+ C_{13}\lambda_n^2(\gamma_T(|a|)+\gamma_{dd^cT}(|a|)).
                \end{array}$$
              The third estimate can be deduced from the choice $$\lambda_n=\frac1{\sqrt{\gamma_T(|a|) +\gamma_{dd^cT}(|a|)}}.$$
        \end{itemize}
    \end{proof}

\section{Pull-Back of positive currents}
    Let $\C^n[0]:=\{(z,L)\in\C^n\times\mathbb{P}^{n-1};\ z\in L\}$ and $T$ be a positive current on $\C^n$. In this section, We study the existence of a positive current $\widehat{T}$ on $\C^n[0]$ such that $\pi_\star \widehat{T}=T$ where $\pi:\C^n[0]\longrightarrow \C^n$ is the canonical projection; in this statement, we say that $T$ admits a blow-up by $\pi$ over 0. We give a positive answer in case of positive plurisubharmonic or plurisuperharmonic currents. Finely, we apply this result to give a second proof of the main result with a supplementary condition in case of positive plurisuperharmonic currents.
    \begin{prop}\label{pro1}
        Let $T\in\mathscr P^+_p(\C^n)$ (resp. $T\in\mathscr P^-_p(\C^n)$). Then, $T$ admits a blow-up $\widehat{T}$ (a positive current  on $\C^n[0]$ such that $\pi_\star\widehat{T}=T$). Furthermore, for $r>0$ one has
        \begin{equation}\label{eq3.1}
            ||\widehat{T}||(\pi^{-1}(B(r)))\leq \nu_T(r)-\nu_T(0)+C_r\nu_T(r).
        \end{equation}
        resp.
        \begin{equation}\label{eq3.2}
            \begin{array}{lcl}
              \ds||\widehat{T}||(\pi^{-1}(B(r))) & \leq  &\ds |\nu_T(r)-\nu_T(0)|+C_r\nu_T(r) -C_r'\nu_{dd^cT}(r) \\
              &  &\ds \hfill +\int_0^r\left(\frac{t^{2p}}{r^{2p}}-1\right)\frac{\nu_{dd^cT}(t)}tdt
            \end{array}
        \end{equation}
        where $C_r:=\sum_{k=1}^pC_p^k r^{2k}$ and $C_r':=\sum_{k=1}^p \frac{C_p^k}{2k}r^{2k}$.
    \end{prop}
    Proposition \ref{pro1} is proved by Giret \cite{Gi} in the case of positive closed currents.
    \begin{proof}
        Let $T\in\mathscr P^\pm_p(\C^n)$. The canonical projection $\pi$ is a submersion from $\C^n[0]\smallsetminus\pi^{-1}(\{0\})$ to $\C^n\smallsetminus\{0\}$, so $\mathcal T:=\pi^\star(T_{|\C^n\smallsetminus\{0\}})$ exists. Furthermore  $\mathcal T$ has a locally finite mass near every point of $\pi^{-1}(\{0\})$. Let $\widehat{T}$ be the trivial extension of $\mathcal T$ by zero over $\mathbb P^{n-1}\approx\{0\}\times\mathbb P^{n-1}=\pi^{-1}(\{0\})$, $\widehat{T}$ is positive on $\C^n[0]$ and $\pi_\star\widehat{T}=T$. To prove Inequality (\ref{eq3.1}) (resp. (\ref{eq3.2})), let $\omega$ be the K\"ahler form of $\C^n[0]$. Then, for $0<\epsilon<r$, we have
        $$\begin{array}{lcl}
             \ds||\mathcal T||(\pi^{-1}(B(\epsilon,r)))& = &\ds\int_{B(\epsilon,r)}T\w\pi_\star\omega^p =\int_{B(\epsilon,r)}T\w(\alpha+\beta)^p\\
             & = &\ds \sum_{k=0}^pC_p^k\int_{B(\epsilon,r)}T\w\alpha^{p-k}\w\beta^k\\
             & = &\ds \sum_{k=0}^{p-1}C_p^k\int_{B(\epsilon,r)}T\w\beta^k\w\alpha^{p-k}+\int_{B(\epsilon,r)}T\w\beta^p.
          \end{array}$$
          \begin{itemize}
            \item \emph{First case: $T\in\mathscr P_p^+(\C^n)$.} For every $0\leq k\leq p-1$, Lelong-Jensen formula applied to the current $T\w \beta^k$ gives
                $$\begin{array}{lcl}
                    \ds \int_{B(\epsilon,r)}T\w\beta^k\w\alpha^{p-k}& =& \ds\frac1{r^{2(p-k)}}\int_{B(r)} T\w\beta^p -\frac1{\epsilon^{2(p-k)}} \int_{B(\epsilon)}T\w\beta^p\\
                    & &\hfill \ds-\int_\epsilon^r\left(\frac1{t^{2(p-k)}} -\frac1{r^{2(p-k)}} \right)tdt\int_{B(t)}dd^cT\w\beta^{p-1}\\
                    & &\hfill \ds-\left(\frac1{\epsilon^{2(p-k)}}-\frac1{r^{2(p-k)}}\right)\int_0^\epsilon tdt\int_{B(t)} dd^cT\w \beta^{p-1}\\
                    &\leq& r^{2k}\nu_T(r)-\epsilon^{2k}\nu_T(\epsilon)
                \end{array}$$
                hence,
                $$\begin{array}{lcl}
                        \ds||\mathcal T||(\pi^{-1}(B(\epsilon,r)))& = &\ds r^{2p}\nu_T(r)-\epsilon^{2p}\nu_T(\epsilon) +\sum_{k=0}^{p-1}C_p^k\int_{B(\epsilon,r)}T\w\beta^k\w\alpha^{p-k}\\
                        & \leq  & \ds \sum_{k=0}^pC_p^k\left(r^{2k}\nu_T(r)-\epsilon^{2k}\nu_T(\epsilon)\right)\\
                        & \leq & \ds \nu_T(r)-\nu_T(\epsilon)+\sum_{k=1}^pC_p^k\left(r^{2k}\nu_T(r)-\epsilon^{2k}\nu_T(\epsilon)\right).
                    \end{array}$$
                which is bounded independently to $\epsilon$. Inequality (\ref{eq3.1}) is obtained by tending $\epsilon$ to 0.
            \item \emph{Second case: $T\in\mathscr P_p^-(\C^n)$.} For every $0\leq k< p$, if we set
                $$\begin{array}{lcl}
                    \Lambda_k(r)&:=&\ds \nu_{T\w\beta^k}(r)+\int_0^r\left(\frac{t^{2(p-k)}}{r^{2(p-k)}}-1 \right)\frac{\nu_{dd^cT\w\beta^k}(t)}t dt\\
                    & = &\ds r^{2k}\nu_T(r)+ \int_0^r\left(\frac{t^{2(p-k)}}{r^{2(p-k)}}-1 \right)t^{2k}\frac{\nu_{dd^cT}(t)}t dt\\
                    & =: & r^{2k}\nu_T(r)+ \mathscr J_k(r)
                  \end{array}$$
                then $\Lambda_k$ is a non-negative increasing function and $\nu_T(0)=\lim_{r\to0}\Lambda_0(r)$. Therefore, as in the previous case, one has
                $$\begin{array}{lcl}
                        \ds||\mathcal T||(\pi^{-1}(B(\epsilon,r)))& = &\ds r^{2p}\nu_T(r)-\epsilon^{2p}\nu_T(\epsilon) +\sum_{k=0}^{p-1}C_p^k(\Lambda_k(r)-\Lambda_k(\epsilon))\\
                        & =  & \ds \sum_{k=0}^pC_p^k\left[(r^{2k}\nu_T(r)-\epsilon^{2k}\nu_T(\epsilon)) +(\mathscr J_k(r)-\mathscr J_k(\epsilon))\right].
                    \end{array}$$
                If we tend $\epsilon$ to 0 and using the fact that  $\nu_{dd^cT}$ is a non-positive decreasing function, we obtain
                $$\begin{array}{lcl}
                        \ds||\widehat{T}||(\pi^{-1}(B(r)))& = &\ds||\mathcal T||(\pi^{-1}(B(r)\smallsetminus \{0\}))\\
                        & \leq &\ds |\nu_T(r)-\nu_T(0)|+ \sum_{k=1}^pC_p^k r^{2k}\nu_T(r) +\sum_{k=0}^{p-1}C_p^k\mathscr J_k(r)\\
                        & \leq &\ds |\nu_T(r)-\nu_T(0)|+ \sum_{k=1}^pC_p^k r^{2k}\nu_T(r)\\
                        & & \ds-\sum_{k=1}^{p-1}C_p^k\frac{r^{2k}}{2k}\nu_{dd^cT}(r)+\mathscr J_0(r).
                \end{array}$$
                which completes the proof.
          \end{itemize}
    \end{proof}
    \begin{defn}
        We say that a positive current $S$ satisfies the condition of restriction along an hypersurface $Y$ if for any equation $\{h=0\}$ of $Y$ in a local chart $U$, one has $\log|h|\in L^1(U,\sigma_S)$ where $\sigma_S=S\w\beta^p$ the trace measure associated to $S$.
    \end{defn}
    The problem now is to give a suitable condition on $T$ to have $\widehat{T}$ satisfies the condition of restriction along the hypersurface $\mathbb P^{n-1}$. For this aim, we need the following lemma where we set $$\Delta:=\left\{z\in\C^n;\ |z_j|<1,\ \forall\; j\in\{1,...,n\}\right\}$$ the unit polydisc of $\C^n$ and $\Delta^*=\Delta\smallsetminus \{0\}$.
    \begin{lem}\label{lem5}
        Let $S$ be a positive current of bidimension (p,p) on a neighborhood $\Omega$ of $\Delta$ in  $\C^n$. Then the following conditions are equivalent:
        \begin{enumerate}
          \item $\log|z_k|\in L^1(\Delta^*_k(1),\sigma_S)$.
          \item $\ds\int_0^1 \frac{\sigma_S(\Delta^*_k(r))}rdr<+\infty$, where $\Delta^*_k(r)=\{z\in\Delta;\ 0<|z_k|<r\}$ for all $1\leq k\leq n$.
        \end{enumerate}
    \end{lem}
    This lemma was proved by Raby \cite{Ra} for a positive closed current. We give the same proof (In fact, only the positivity of the current is needed).
    \begin{proof}
        Equivalence between the two conditions is deduced from the following equality:
        \begin{equation}\label{eq3.3}
        \int_0^1 \frac{\sigma_S(\Delta^*_k(r))}rdr=\int_{\Delta^*}-\log|z_k|d\sigma_S.
        \end{equation}
        to show (\ref{eq3.3}), Let $u\in]0,1]$. Then
        $$\int_u^1 \frac{\sigma_S(\Delta^*_k(r))}rdr=\int_u^1\left(\frac1r\int_{\Delta^*_k(r)}d\sigma_S\right)dr =\int_{D_k(u)}d\sigma_S\otimes \frac{dr}r$$
        where $D_k(u)=\Delta^*_k(r)\times]u,1[$.
        Therefore,
        $$\begin{array}{lcl}
            \ds\int_u^1 \frac{\sigma_S(\Delta^*_k(r))}rdr & = & \ds\int_{\Delta^*}\left(\int_{\max(u,|z_k|)}^1\frac{dr}r\right)d\sigma_S \\
            & = &\ds  \int_{\Delta^*}-\log(\max(u,|z_k|))d\sigma_S
          \end{array}$$
        hence,
        $$\int_{\Delta^*\smallsetminus \Delta^*_k(u)}-\log|z_k|d\sigma_S\leq \int_u^1 \frac{\sigma_S (\Delta^*_k(r))}rdr\leq \int_{\Delta^*}-\log|z_k|d\sigma_S.$$
        If we tend $u$ to 0, we obtain equality (\ref{eq3.3}).
    \end{proof}
    \begin{prop}\label{pro2}
        Let $T\in\mathscr P^+_p(\C^n)$ (resp. $T\in\mathscr P^-_p(\C^n)$) such that
        $$\int_0^1\frac{\nu_T(r)- \nu_T(0)}rdr<+\infty$$
        resp.
        $$\int_0^1\frac{|\nu_T(r)- \nu_T(0)|}rdr<+\infty \hbox{ and } \int_0^1\frac{\nu_{dd^cT(r)}}r \log rdr<+\infty.$$
        Then $\widehat{T}$ satisfies the condition of restriction along $\mathbb P^{n-1}$.
    \end{prop}
    This result is due to Giret \cite{Gi} in the case of positive closed currents.
    \begin{proof}$ $
        \begin{itemize}
          \item \emph{First case $T\in\mathscr P_p^+(\C^n)$.}  Thanks to Inequality (\ref{eq3.1}), one has
            $$||\widehat{T}||(\pi^{-1}(B(r))\leq  \nu_T(r)-\nu_T(0)+C_r\nu_T(r)$$
            where $C_r=\sum_{k=1}^pC^k_pr^{2k}$. Thanks to Lemma \ref{lem5}, $\widehat{T}$ satisfies the condition of restriction along $\mathbb P^{n-1}$ if $$\int_0^1\frac{\nu_T(r)- \nu_T(0)}rdr<+\infty.$$
          \item \emph{Second case $T\in\mathscr P_p^-(\C^n)$.} As in the previous case, thanks to Inequality (\ref{eq3.2}), we have
            $$||\widehat{T}||(\pi^{-1}(B(r)) \leq |\nu_T(r)-\nu_T(0)|+C_r\nu_T(r)-C'_r\nu_{dd^cT}(r)+\mathscr J_0(r)$$
            where $C'_r=\sum_{k=1}^p\frac{C^k_p}{2k} r^{2k}$. Thanks to Lemma \ref{lem5}, $\widehat{T}$ satisfies the condition of restriction along $\mathbb P^{n-1}$ if $$\int_0^1\frac{|\nu_T(r)- \nu_T(0)|}rdr<+\infty\quad \hbox{and }\quad \int_0^1\mathscr J_0(r)dr<+\infty.$$
            A simple computation shows that
            $$\begin{array}{lcl}
                \ds\int_0^1 \mathscr J_0(r)dr& = &\ds \int_0^1\frac1r\left(\int_0^r\left(\frac{t^{2p}}{r^{2p}}-1\right) \frac{\nu_{dd^cT}(t)}tdt\right)dr \\
                & = &\ds \int_0^1\frac{\nu_{dd^cT}(t)}t\left(\log t-\frac{t^{2p}}{2p}+\frac1{2p}\right)dt.
              \end{array}$$
        \end{itemize}
    \end{proof}
    As an application of proposition \ref{pro2}, we give a second proof of the Main result.
    \begin{cor}\label{cor1}
        Let $T$ be a positive plurisubharmonic or plurisuperharmonic current as in proposition \ref{pro2}. Then $T$ admits a tangent cone at 0.
    \end{cor}
    \begin{proof}
        Let $\mu:\C^n\smallsetminus\{0\}\to\mathbb P^{n-1}$ defined by $\mu(z)=[z]$. Thanks to proposition \ref{pro2}, the current $\mu^\star\left(\widehat{T}_{|\mathbb P^{n-1}}\right)$ is positive on $\C^n\smallsetminus\{0\}$ and it admits a trivial extension $\Theta_T$ on $\C^n$. We prove that $\Theta_T=\lim_{a\to0}h_a^\star T$ (see \cite{Gi}) so $\Theta_T$ is the tangent cone to $T$ at 0.
    \end{proof}

\section{Appendix: Conic currents}
    Let $T$ be a positive plurisubharmonic or plurisuperharmonic current of bidimension $(p,p)$ on $\C^n$. recall that $T$ is called \emph{conic} if $h_a^\star T=T$ for every $a\in\C^*$. It is well known that $dd^c(h_a^\star T)=h_a^\star (dd^cT)$ so if  $T$ is conic then $dd^cT$ is also conic and the two functions $\nu_T$ and $\nu_{dd^cT}$ are constant. In particular, if $T\in\mathscr P_p^\pm(\C^n)$ then $\nu_{dd^cT}\equiv \nu_{dd^cT}(0)=0$ so $T$ is pluriharmonic. The following lemma gives more informations.
    \begin{lem}\label{lem6}
        Let $T\in\mathscr P_p^\pm(\C^n)$.
        The following assertions are equivalent:
        \begin{enumerate}
          \item $T$ is invariant by dilatations $h_a$ for all $a\in\C^*$;
          \item $T$ is invariant by dilatations $h_a$ for all $a\in]0,+\infty[$;
          \item $T$ is pluriharmonic and $T\w\alpha^p=0$ on $\C^n\smallsetminus\{0\}$;
          \item $T$ is the extension to $\C^n$ of the pull-back of a positive current by the projection $\mu:\C^n\smallsetminus\{0\}\to\mathbb P^{n-1}$.
        \end{enumerate}
    \end{lem}
    \begin{proof}
        It's clear that (1) implies (2). With the hypothesis of (2) we have $dd^cT$ is also invariant by dilatations $h_a$ for all $a\in]0,+\infty[$ so $\nu_T$ and $\nu_{dd^cT}$ are constants. Thanks to Lemma \ref{lem1}, one has $$\int_{B(\epsilon,r)}T\w\alpha^p=0,\quad \forall\;0<\epsilon<r.$$
        (3) implies (4) and (4) implies (1) are proved by Haggui in \cite{Ha}.
    \end{proof}
    \begin{rem}
        The current $T_0$ of Example \ref{exple1} is positive plurisuperharmonic conic non pluriharmonic, so if we study positive plurisuperharmonic current non satisfying condition $(C)_0$ then assertion $(3)$ in Lemma \ref{lem6} may be replaced by         \textit{$(3)'$ $dd^cT$ is conic and
        $$\int_{B(\epsilon,r)}T\w\alpha^p=\nu_{dd^cT}(0)\log\frac\epsilon r,\quad \forall\; 0<\epsilon<r.$$}
    \end{rem}
    \begin{prop}\label{pro3}
        If $T\in\mathscr P_p^+(\C^n)$ (resp. $T\in\mathscr P_p^-(\C^n)$) then every adherence value of $(h_a^\star T)_a$ is a positive  conic pluriharmonic current on $\C^n$.
    \end{prop}
    \begin{proof}
        Let $\Theta=\lim_{k\to+\infty}h_{a_k}^\star T$ where $a_k\underset{k\to+\infty}\longrightarrow0$.
        \begin{itemize}
          \item \emph{First case $T\in\mathscr P_p^+(\C^n)$.}
            Using $h_{a_k}^\star T$ instead of $T$, the Lelong-Jensen formula gives, for every $0<\epsilon<r$ and $k$,
            $$\begin{array}{lcl}
                \nu_T(|a_k|r)-\nu_T(|a_k|\epsilon)& = &\ds \int_\epsilon^r\left(\frac1{t^{2p}} -\frac1{r^{2p}} \right)t^{2p-1}\nu_{dd^c(h_{a_k}^\star T)}(t)dt\\
                &  & \ds +\left(\frac1{\epsilon^{2p}}-\frac1{r^{2p}}\right) \int_0^\epsilon t^{2p-1}\nu_{dd^c(h_{a_k}^\star T)}(t)dt\\
                & &\hfill  \ds+\int_{B(\epsilon,r)}h_{a_k}^\star T\w\alpha^p.
            \end{array}$$
            If $k\to+\infty$, we obtain
            $$\begin{array}{lcl}
                0 & = &\ds \int_{B(\epsilon,r)}\Theta\w\alpha^p+\int_\epsilon^r\left(\frac1{t^{2p}} -\frac1{r^{2p}} \right)t^{2p-1}\nu_{dd^c\Theta}(t)dt\\
                & & \hfill  \ds +\left(\frac1{\epsilon^{2p}}-\frac1{r^{2p}}\right) \int_0^\epsilon t^{2p-1}\nu_{dd^c\Theta}(t)dt.
              \end{array}$$
            $\Theta$ is positive plurisubharmonic, hence the three terms of the previous equality are equal to zero. In particular $\Theta$ is pluriharmonic and $\Theta\w\alpha^p=0$ on $\C^n\smallsetminus\{0\}$. Thanks to Lemma \ref{lem6}, $\Theta$ is conic.
          \item \emph{Second case $T\in\mathscr P_p^-(\C^n)$.}
            Like in the previous case, we consider the non-negative increasing function  $$\Lambda_T(r)=\nu_T(r)+\int_0^r\left(\frac{t^{2p}}{r^{2p}}-1\right)\frac{\nu_{dd^cT}(t)}{t}dt.$$
            We remark that $h_a^\star T\in\mathscr P_p^-(\C^n)$ because
            $$\int_0^{r_0}\frac{\nu_{dd^c(h_a^\star T)}(t)}{t}dt = \int_0^{|a|r_0} \frac{\nu_{dd^cT}(t)}{t}dt>-\infty,\quad \forall\; a\in\C^*$$
            and
            $$\begin{array}{lcl}
                \Lambda_{h_a^\star T}(r)  & = &\ds \nu_T(|a|r)+\int_0^r\left(\frac{t^{2p}}{r^{2p}}-1\right) \frac{\nu_{(h_a^\star dd^cT)}(t)}{t}dt \\
                & = &\ds \nu_T(|a|r)+\int_0^r\left(\frac{t^{2p}}{r^{2p}}-1\right) \frac{\nu_{dd^cT}(|a|t)}{t}dt \\
                & = & \Lambda_T(|a|r).
              \end{array}$$
            Thanks to the proof of theorem \ref{th2}, for every $0<\epsilon<r$ and $k$ (large enough),
            $$\begin{array}{lcl}
                \ds\Lambda_{h_{a_k}^\star T}(r)- \Lambda_{h_{a_k}^\star T}(\epsilon)& = & \ds\Lambda_T(|a_k|r)-\Lambda_T(|a_k|\epsilon)\\
                 & = &\ds \int_{B(\epsilon,r)}h_{a_k}^\star T\w\alpha^p
              \end{array}$$
            If $k\to+\infty$, we obtain $\Lambda_\Theta$ is constant and $\Theta\w\alpha^p=0$ on $\C^n\smallsetminus\{0\}$.\\
            So, $\Lambda_\Theta(r)=\nu_\Theta(0)$ for every $r>0$  and this can be written as
            \begin{equation}\label{eq4.1}
                \nu_\Theta(r)+\int_0^r\left(\frac{t^{2p}}{r^{2p}}-1\right)\frac{\nu_{dd^c\Theta}(t)}{t}dt =\nu_\Theta(0),\quad \forall\; r>0.
            \end{equation}
            Furthermore, one has
            \begin{equation}\label{eq4.2}
                \nu_\Theta(r)=\lim_{k\to+\infty} \nu_{h_{a_k}^\star T}(r)= \lim_{k\to+\infty} \nu_T(|a_k|r)=\nu_\Theta(0),\quad \forall\; r>0.
            \end{equation}
            Equalities (\ref{eq4.1}) and (\ref{eq4.2}) give
            $$\int_0^r\left(\frac{t^{2p}}{r^{2p}}-1\right)\frac{\nu_{dd^c\Theta}(t)}{t}dt =0,\quad \forall\; r>0.$$
            Since $\Theta$ is a positive plurisuperharmonic current, so  $\nu_{dd^c\Theta}$ is non positive, then $\nu_{dd^c\Theta}\equiv0$. Hence $\Theta$ is a positive pluriharmonic current satisfying $\Theta\w\alpha^p=0$ on $\C^n\smallsetminus\{0\}$, thanks to lemma \ref{lem6}, $\Theta$ is conic.
        \end{itemize}
    \end{proof}

    \begin{cor}\label{cor2}
        Let $T\in\mathscr P^\pm(\C^n)$ and $(a_k)_k,\ (b_k)_k$ are two sequences of complex numbers such that $\left|\frac{a_k}{b_k}\right|$ and $\left|\frac{a_k}{b_k}\right|$ are bounded. If $h_{a_k}^\star T$ and $h_{b_k}^\star T$ converge weakly then $h_{a_k}^\star T-h_{b_k}^\star T$ converges weakly to 0.
    \end{cor}
    Therefore, the set of adherent values of $(h_a^\star T)_a$ does not change  if we restrict to the case $a\in]0,+\infty[$.
    \begin{proof}
        Let $\Theta_1=\lim_{k\to+\infty}h_{a_k}^\star T$ and $\Theta_2=\lim_{k\to+\infty}h_{b_k}^\star T$  such that $c_k=\frac{b_k}{a_k}\underset{k\to+\infty}\longrightarrow c\in\C^*$ (we extract  subsequences if necessary). For every $\varphi\in\mathscr D_{p,p}(\C^n)$, we have
        $$\langle h_{a_k}^\star T,\varphi\rangle= \langle h_{1/c_k}^\star h_{b_k}^\star T,\varphi\rangle=\langle h_{b_k}^\star T,h_{c_k}^\star\varphi\rangle\underset{k\to+\infty}\longrightarrow \langle \Theta_2,h_c^\star\varphi\rangle$$
        hence, $\langle \Theta_1,\varphi\rangle=\langle h_{1/c}^\star\Theta_2,\varphi\rangle=\langle \Theta_2,\varphi\rangle$, because $\Theta_2$ is conic.
    \end{proof}

\end{document}